\documentclass[12pt]{amsart}
\usepackage{amscd,verbatim}

\renewcommand{\L}{\mathbb{L}}
\newcommand{\Q}{\mathbf{Q}}
\newcommand{\Z}{\mathbf{Z}}

\renewcommand{\epsilon}{\varepsilon}

\newcommand{\End}{\operatorname{End}}
\newcommand{\Hom}{\operatorname{Hom}}

\newtheorem{thm}{Theorem}

\newtheorem{prop}{Proposition}
\theoremstyle{definition}

\theoremstyle{remark}
\newtheorem{rk}{Remark}

\begin{document}
\title{Smash-nilpotent cycles on abelian $3$-folds}
\author{Bruno Kahn}
\address{Institut de Math\'ematiques de Jussieu\\175--179, rue du Chevaleret\\75013 Paris\\France}
\email{kahn@math.jussieu.fr}
\author{Ronnie Sebastian}
\address{Tata Institute of Fundamental Research\\Homi Bhabha Road\\Mumbai 400 005\\India}
\email{ronnie@math.tifr.res.in}
\date{April 30, 2009}
\begin{abstract}  We show that homologically trivial algebraic cycles on a $3$-dimensional
abelian variety are smash-nilpotent.
\end{abstract}
\subjclass[2000]{14C15, 14K05.}
\maketitle

\section*{Introduction} Let $X$ be a smooth projective variety over a field $k$. An algebraic
cycle $Z$ on $X$ (with rational coefficients) is \emph{smash-nilpotent} if there exists $n>0$
such that $Z^n$ is rationally equivalent to $0$ on $X^n$. Smash-nilpotent cycles have the
following properties:

	\begin{enumerate}
\item The sum of two smash-nilpotent cycles is smash-nilpotent.
	\item The subgroup of smash-nilpotent cycles forms an ideal under the intersection product as
$(x\cdot y) \times (x\cdot y) \cdots \times (x\cdot y) = (x\times x \times \cdots \times
x)\cdot (y\times y \times \cdots \times y)$.
	\item On an abelian variety, the subgroup of smash-nilpotent cycles forms an ideal under the
Pontryagin product as $(x*y) \times (x*y) \times \cdots \times (x*y) = (x\times x \times \cdots
\times x)*(y\times y \times \cdots \times y)$ where $*$ denotes the Pontryagin product.
	\end{enumerate}

Voevodsky \cite[Cor. 3.3]{voe} and
Voisin \cite[Lemma 2.3]{voisin} proved that any cycle algebraically equivalent to $0$ is
smash-nilpotent. On the other hand, because of cohomology, any smash-nilpotent cycle is
numerically equivalent to $0$; Voevodsky conjectured that the converse is true \cite[Conj.
4.2]{voe}.

This conjecture is open in general. The main result of this note is:

\begin{thm}\label{t0} Let $A$ be an abelian variety of dimension $\le 3$. Any homologically
trivial cycle on $A$ is smash-nilpotent.
\end{thm}

In characteristic $0$ we can improve ``homologivally trivial" to ``numerically trivial",
thanks to Lieberman's theorem \cite{lieberman}.

Nori's results in \cite{nori} give an example of a group of  smash-nilpotent cycles which is  not finitely generated modulo algebraic equivalence. The proof of Theorem \ref{t0} actually gives the uniform bound $21$ for the degree of smash-nilpotence on this group, see Remark \ref{r1}. 
See Proposition \ref{t2} for partial results in dimension $4$. 

\section{Beauville's decomposition, motivically}\label{bmot}

For any smooth projective variety $X$ and any integer $n\ge 0$, we write as in \cite{B2}
$CH^n_\Q(X)=CH^n(X)\otimes \Q$, where $CH^n(X)$ is the Chow group of cycles of codimension $n$
on $X$ modulo rational equivalence.  

Let $A$ be an abelian variety of dimension $g$. For
$m\in\Z$, we denote  by
$\langle m\rangle$ the endomorphism of multiplication by
$m$ on
$A$, viewed as an algebraic correspondence. In \cite{B2}, Beauville introduces an eigenspace
decomposition of the rational Chow groups of $A$ for the actions of the operators $\langle
m\rangle$, using the Fourier transform. Here is an equivalent definition: in the category of
Chow motives with rational coefficients, the endomorphism $1_A\in
\End(h(A))=CH^g_\Q(A\times A)$ is given by the class of  the diagonal $\Delta_A$. We have
the canonical Chow-K\"unneth decomposition of Deninger-Murre
\[1_A=\sum_{i=0}^{2g} \pi_i\]
\cite[Th. 3.1]{denmur}, where the $\pi_i$ are orthogonal idempotents and $\pi_i$ is characterised by
$\pi_i\langle m\rangle^* = m^i\pi_i$  for any $m\in\Z$. This yields a canonical Chow-K\"unneth decomposition
of the Chow motive $h(A)$ of $A$:
\[h(A)=\bigoplus_{i=0}^{2g} h^i(A), \quad h^i(A)= (A,\pi_i)\]
(see \cite[Th. 5.2]{scholl}). Then, under the isomorphism
\[CH^n_\Q(A) = \Hom(\L^n,h(A))\]
(where $\L$ is the Lefschetz motive) we have
\[CH^n(A)_{[r]}:= \{x\in CH^n_\Q(A)\mid \langle m\rangle^* x= m^r x\; \forall m\in\Z\} =
\Hom(\L^n,h^r(A)).\]

\begin{rk} In Beauville's notation, we have
\[CH^n(A)_{[r]} = CH^n_{2n-r}(A).\]

We shall use his notation in \S \ref{s3}.
\end{rk}

\section{Skew cycles on abelian varieties}
Let $\beta\in CH^*_\Q(A)$. Assume that $\langle -1\rangle^*\beta = -\beta$:  we say that $\beta$ is
\emph{skew}. This implies that $\beta$ is homologically equivalent  to $0$. 

For $g\le 2$, the Griffiths group of $A$ is $0$ and there is nothing to prove. For $g=3$, the
Griffiths group of $A$ is a quotient of
$CH^2(A)_{[3]}$ \cite[Prop. 6]{B2}; thus
Theorem \ref{t0} follows from the more general

\begin{prop}\label{t1} Any skew cycle on an abelian variety is smash-nil\-po\-tent.
\end{prop}

This applies notably to the Ceresa cycle \cite{ceresa}, for any genus.

\begin{proof} We may assume $\beta$ homogeneous, say, $\beta\in CH^n_\Q(A)$. 
View $\beta$ as a morphism $\L^n\to h(A)$ in the category of Chow motives.  Thus, for all $i$:
\[-\pi_i\beta = \pi_i\langle -1\rangle^*\beta = (-1)^i\pi_i\beta\]
hence $\pi_i\beta = 0$ for $i$ even.

This shows that $\beta$ factors through a morphism
\[\tilde\beta:\L^n\to h^{odd}(A)\]
with $h^{odd}(A) = \bigoplus_{i \text{ odd}} h^i(A)$. 

But $\L^n$ is evenly finite-di\-men\-sion\-al and $h^{odd}(A)$ is oddly finite-dimen\-sional in
the sense of  S.-I. Kimura. (Indeed, $S^{2g+1}(h^1(A))=h^{2g+1}(A)=0$ by \cite[Theorem]{shermenev},  and a direct summand of an odd tensor power of an oddly finite-dimensional motive is oddly finite dimensional by \cite[Prop. 5.10 p. 186]{kimura}.) Hence the conclusion follows from \cite[prop. 6.1 p.
188]{kimura}. \end{proof}

\begin{rk}\label{r1}  Kimura's proposition 6.1 shows in fact that all
$z\in\break \Hom(\L^n,h^{odd}(A))$ verify $z^{\otimes N+1} = 0$ for a fixed $N$, namely, the
sum of the odd Betti numbers of $A$. If $z\in \Hom(\L^n,h^i(A))$ for some odd $i$, then we may
take for $N$ the $i$-th Betti number of $A$. Thus, for $i=3$ and if $A$ is a $3$-fold, we find
that all $z\in \Hom(\L,h^3(A))$ verify $z^{\otimes 21}=0$.
\end{rk}

\section{The $4$-dimensional case}\label{s3}

\begin{prop}\label{t2}  If $g=4$, homologically trivial cycles on $A$, except perhaps those
which occur in parts $CH^2_0(A)$ or $CH^3_2(A)$ of the Beauville decomposition, are
smash-nilpotent.
\end{prop}

\begin{proof} Let $A$ be an abelian variety and let $\hat{A}$ denote its dual abelian variety. 
We know, from \cite{B2}, the following:
\begin{enumerate}
		\item $CH^{p}_{s}(A)=0$ for $p \in \{0,1,g-2,g-1,g\}$ and $s<0$. \cite[Prop. 3a]{B2}.
	\item $CH^{p}_{p}(A)$ and $CH^{g}_{s}(A)$ consist of cycles algebraically equivalent to 0 for
all values of $p$ and all values of $s>0$. \cite[Prop. 4]{B2}.
\end{enumerate}

For $g=4$, using these results and Proposition \ref{t1} one can conclude smash
nilpotence for homologically trivial cycles which are not in $CH^2_0(A)$ or $CH^3_2(A)$. Note that, with the notation of \S \ref{bmot},
\[CH^3_2(A)= \Hom(\L^3,h^4(A)),\quad CH^2_0(A) = \Hom(\L^2,h^4(A)).\]

In the
case of $CH^2_0(A)$, the problem is whether there are any homologically
trivial cycles: in view of the above expression, this is conjecturally not the case, cf.
\cite[Prop. 5.8]{jannsen}.
	\end{proof}	

\begin{rk} Some of these arguments also follow from a paper of Bloch and Srinivas
\cite{BS}.
\end{rk}

\section*{Acknowledgements} The second author would like to thank Prof. J.P. Murre for his course on Motives at TIFR in January 2008, from where he learnt  about the conjecture. Both authors would like to thank V. Srinivas for discussions leading to this result. The second author is being funded by CSIR.


\begin{thebibliography}{II}
\bibitem{B2} A. Beauville {\it  Sur l'anneau de Chow d'une vari\'et\'e ab\'elienne}, Math.
Ann. {\bf 273}  (1986), 647--651.
\bibitem{BS} S. Bloch, V. Srinivas {\it Remarks on Correspondences and Algebraic Cycles},
Amer. J. Math.  {\bf 105}  (1983), 1235--1253.
\bibitem{ceresa} G. Ceresa {\it $C$ is not algebraically equivalent to $C\sp{-}$ in its
Jacobian},  Ann. of Math. {\bf 117}  (1983), 285--291.
\bibitem{denmur} C. Deninger, J.P. Murre {\it Motivic decomposition of abelian shecmes and the Fourier transform}, J. reine angew. Math. {\bf 422}Ê(1991), 201--219.
\bibitem{jannsen} U. Jannsen {\it Motivic sheaves and filtrations on Chow groups}, 
{\it in}  Motives, Proc. Sympos. Pure Math. {\bf 55} 
(1), Amer. Math. Soc., 1994,  245--302.
\bibitem{kimura} S.-I. Kimura {\it Chow groups are finite dimensional, in a sense}, 
Math. Ann. {\bf 331} (2005), 173--201.
\bibitem{lieberman} D. I. Lieberman {\it Numerical and homological equivalence of algebraic
cycles on Hodge manifolds},  Amer. J. Math.  {\bf 90} (1968), 366--374.
\bibitem{nori} M. Nori {\it Cycles on the generic abelian threefold}, 
Proc. Indian Acad. Sci. Math. Sci. {\bf 99} (1989), 191--196. 
\bibitem{shermenev} A. M. Shermenev {\it The motive of an abelian variety}, Funct. Anal. Appl. 
{\bf 8} (1974), 47--53.
\bibitem{scholl} A. Scholl {\it Classical motives}, {\it in}  Motives, Proc. Sympos. Pure Math. {\bf 55} (1), Amer. Math. Soc., 1994,  163--187.
\bibitem{voe} V. Voevodsky {\it A nilpotence theorem for cycles algebraically equivalent 
to $0$},  Internat. Math. Res. Notices  {\bf 1995}, 187--198.
\bibitem{voisin} C. Voisin {\it Remarks on zero-cycles of self-products of varieties}, {\it
in}  Moduli of vector bundles, Lect. Notes in Pure Appl. Math.
{\bf 179}, Dekker, 1996,  265--285. 
\end{thebibliography}
\end{document}